\documentclass[12pt]{amsart}
\usepackage[latin1]{inputenc}
\usepackage[english]{babel}
\usepackage{amsmath,amssymb,amsthm,amsfonts, color}
\usepackage{enumerate}
\usepackage{graphicx}
\usepackage{latexsym}
\usepackage[all]{xy}
\usepackage{tikz}

\makeatletter
\@namedef{subjclassname@2010}{ \textup{2010} Mathematics Subject Classification}
\makeatother

\newtheorem{theorem}{Theorem}
\newtheorem{corollary}{Corollary}
\newtheorem{propos}{Proposition}
\newtheorem{rem}{Remark}

\newtheorem{lemma}{Lemma}

\theoremstyle{definition}

\begin{document}
\title{Disjointly homogeneous Orlicz spaces revisited.}
\thanks{{\rm *}The author has been supported the Ministry of Science and Higher Education of the Russian Federation (project 1.470.2016/1.4) and by the RFBR grant 18--01--00414.}
\author[Astashkin]{Sergey V. Astashkin}
\address[Sergey V. Astashkin]{Department of Mathematics, Samara National Research University, Moskovskoye shosse 34, 443086, Samara, Russia
}
\email{\texttt{astash56@mail.ru}}
\maketitle

\vspace{-7mm}

\begin{abstract}
Let $1\le p\le\infty$. A Banach lattice $X$ is said to be $p$-disjointly homogeneous or $(p-DH)$ (resp. restricted $(p-DH)$) if every normalized disjoint sequence in $X$ (resp. every normalized sequence of characteristic functions of disjoint subsets) contains a subsequence equivalent in $X$ to the unit vector basis of $\ell_p$. We revisit $DH$-properties of Orlicz spaces and  refine some  previous results of this topic, showing that $(p-DH)$-property is not stable under duality in the class of Orlicz spaces and the classes of restricted $(p-DH)$  and $(p-DH)$ Orlicz spaces are different. Moreover, we give a characterization of uniform $(p-DH)$ Orlicz spaces and establish also closed connections between this property and the duality of $DH$-property.
\end{abstract}


\section{\protect \medskip Introduction.}
\label{Intro}

A Banach lattice $X$ is called {\it disjointly homogeneous} (briefly, $DH$) if two any normalized disjoint sequences in $X$ contain equivalent subsequences. In particular, given a $p\in[1,\infty]$, a Banach lattice $X$ is  {\it $p$-disjointly homogeneous} (briefly, $(p-DH)$) if each  normalized disjoint sequence in $X$ has a subsequence equivalent to the unit vector basis of $\ell_p$ ($c_0$ when $p=\infty$). These notions were introduced explicitly first in the paper \cite{FTT-09} and proved to be very useful in studying the general problem of identifying Banach lattices $X$ such that the ideals of strictly singular and compact operators bounded in $X$ coincide (see also \cite{bib:04}, \cite{HST},  \cite{A-20}, the survey \cite{FHT-survey} and references therein). Results obtained in these papers can be treated as a continuation and substantial development of a classical theorem of V.~D.~Milman \cite{Mil-70} which states that every strictly singular operator in $L_p(\mu)$ has compact square. This was a motivation to find out how large is the class of $DH$ Banach lattices. As it is known, it contains $L_p(\mu)$-spaces, $1 \le p \le\infty$, Lorentz function spaces $L_{q,p}$ and $\Lambda(W, p)$, a certain class of Orlicz function spaces, Tsirelson space and some others. 

Another direction of research was connected with studying the $DH$-property itself and some its versions.
In particular, a special attention has been paid to investigation of the duality problem in the class of $DH$ Banach lattices (a systematic study of this subject was undertaken in the work \cite{FHSTT}). First, if $X$ is an $(\infty-DH)$ Banach lattice, then $X^*$ is a $(1-DH)$ Banach lattice as well \cite[Theorem~2.2]{FTT-09}. In contrast to that, in the same paper it was showed  that the Lorentz space $L_{p,1}[0,1]$, $1<p<\infty$, is $(1-DH)$ but the dual   $L_{q,\infty}[0,1]$, $1/p+1/q=1$, fails to be an $(\infty-DH)$ space. What concerns with Orlicz spaces, in \cite[p.~5863]{FHSTT} and \cite[p.~6]{FHT-survey}, basing on  \cite[Theorem 4.1]{bib:04}, the authors have asserted that if the Orlicz space  $L_F$ is $DH$, then $L_F^*$ is $DH$ as well.  Moreover, in \cite[p.~5877]{FHSTT} (see also Question~3 in the survey \cite{FHT-survey}) it was asked if the dual to a reflexive $p$-DH symmetric space on $[0,1]$ is $DH$. This question was answered in negative recently in \cite{A19a}. 

One more issue relates is a weaker property of restricted disjoint homogeneity that was introduced (for $p=2$) in the paper \cite{HST}. A symmetric space $X$ on $[0,1]$ is said to be {\it restricted $p$-disjointly homogeneous} (in brief, restricted $(p-DH)$) if every sequence of normalized disjoint characteristic functions contains a subsequence equivalent to the unit vector basis of $\ell_p$. Clearly, each $(p-DH)$ space is restricted $(p-DH)$. In \cite{HST}, the authors have proved the converse for Orlicz spaces \cite[Theorem~5.1]{HST} and also asked whether a symmetric space $X$ on $[0,1]$, which is restricted $(p-DH)$, must be $(p-DH)$. This question (repeated also in \cite[p.~19]{FHT-survey}) was  motivated by the fact that restricted $(p-DH)$ spaces have rather "good"\;properties (see \cite{HST} and \cite{A19b}). Answering it, in \cite{A19b}, for every $1\le p<\infty$, various examples of restricted $(p-DH)$ symmetric spaces that are not $(p-DH)$ were given.

The first purpose of this paper is to revisit $DH$-properties of Orlicz spaces and to refine some  previous results of this topic proved in the papers  \cite{bib:04,FHSTT,FHT-survey,HST}. Unfortunately, the proof of Theorem 4.1 from \cite{bib:04} contains a gap, and moreover, as we will see, this result and some others, which used it, turned to be not true in full generality. In particular, if $1<p<\infty$, the classes of restricted $(p-DH)$ and $(p-DH)$ Orlicz spaces are in fact different (see Proposition~\ref{prop1}, Theorem~\ref{th3new} and Corollary~\ref{restricted and not}). Furthermore, $(p-DH)$- property is not stable under duality in the class of Orlicz spaces both in the reflexive ($1<p<\infty$) and non-reflexive ($p=1$) cases (see Theorems~\ref{th2} and \ref{th2new}). Roughly speaking, the results obtained here indicate that $DH$-properties of Orlicz spaces are much richer and more non-trivial than one can see in the above papers. A crucial role in recovering these issues will be played by an example of the Orlicz function given in Proposition~\ref{prop3}. 


Our another ("positive")\:aim is to study the uniform $DH$-property for Orlicz spaces, which means that the constant of equivalence of subsequences in the definition of $DH$-property can be chosen uniformly for all normalized disjoint sequences. We give a characterization of uniform $(p-DH)$ Orlicz spaces in Theorem~\ref{prop2} when $1\le p<\infty$, and then in Corollary~\ref{cor2new} for $p=\infty$. From these results it follows that, for every $1\le p<\infty$, there is a $(p-DH)$ Orlicz space, which is not uniformly $(p-DH)$ (see also Corollary~\ref{th3}). In a sharp contrast to that, each $(\infty-DH)$ Orlicz space is uniformly $(\infty-DH)$ as well (Corollary~\ref{cor2new}). We establish also closed connections between this property and the duality of  $DH$-property in Theorem~\ref{th1} (in the non-reflexive case) and in Theorem~\ref{th1new} (in the reflexive case). Observe that analogous results in a special case of Banach lattices ordered by basis have been obtained earlier in the paper \cite{FHSTT}.

The author would like to thank Professors D.~Freeman, F.~L.~Hernandez, and E.~M.~Semenov for useful discussions on various issues of the topic.

\section{\protect Preliminaries}\label{Prel}

In this section, we will briefly list the definitions and notions used throughout this paper. For more detailed information, we refer to the monographs \cite{LT1,LT-79,KPS,KR,Mal}.

\subsection{\protect Symmetric spaces, Orlicz functions and spaces}
\label{Orlicz}
A {\it Banach function  lattice} $X$ on the interval $[0,1]$ is a Banach space of real--valued Lebesgue measurable functions (of equivalence classes) defined 
on $[0,1]$, which satisfies the ideal property: if $x$ is a measurable function, $|x| \leq |y|$ almost everywhere (a.e.) with respect to the Lebesgue measure $m$ on $[0,1]$ and $y \in X$, then $ x\in X$ and $\|x\|_X \leq \|y\|_X$.

A Banach function  lattice $X$ on $[0,1]$ is said to be a {\it symmetric} ({or \it rearrangement invariant}) space if from the conditions:  $y\in X$, $x$ is a measurable function on $[0,1]$, with $x^*(t)=y^*(t)$, $0<t\le 1$, it follows that $x\in X$ and $\|x\|_{X} =\|y\|_{X}$. Here, $x^*$ denotes the non-increasing, right-continuous rearrangement of a measurable function $x$ on $[0,1]$ given by
$$
x^{*}(t):=\inf \{~s\ge 0:\,m\{u\in [0,1]: |x(u)|>s\}\le t~\},\quad t>0.
$$

For any symmetric space  $X$ on $[0,1]$ we have $L_\infty [0,1]\subseteq X\subseteq L_1[0,1]$. The fundamental function $\phi_X$ of a symmetric space $X$ is defined by  $\phi_X(t):=\|\chi_{[0,t]}\|_X$, $0\le t\le 1$. In what follows, $\chi_A$ is the characteristic function of a set $A$. If $X$ is a symmetric space, then $X^0$ is the closure of $L_\infty$ in $X$. Then $X^0$ is a symmetric space, which is separable provided that $X\ne L_\infty$.

The {\it K\"othe dual} (or the {\it associated} space) $X'$ of a symmetric space $X$  consists of all measurable functions $y$ such that
$$
\Vert y\Vert _{X'}:= \sup \Big\{\int _{0}^1\vert
x(t)y(t)\vert dt:\ x\in X,\ \Vert x \Vert _{X}\leq 1\Big\}<\infty.
$$
If $X$ is separable, then $X'$ coincides with the (Banach) dual space $X^*$. 

The most known and important symmetric spaces are the $L_p$-spaces, $1\le p\le\infty$. Their natural generalization is the Orlicz spaces. Let $F$ be an Orlicz function, that is, an increasing convex function on $[0, \infty)$ such that $F(0) = 0$ and $\lim_{t\to\infty}F(t)=\infty$. Denote by $L_{F}:=L_F[0,1]$ the {\it Orlicz space}  
on $[0,1]$ (see e.g. \cite{KR}) endowed with the Luxemburg--Nakano norm
$$
\| f \|_{L_{F}} = \inf \{\lambda > 0 \colon \int_0^1 F(|f(t)|/\lambda) \, dt \leq 1 \}.
$$
In particular, if $F(u)=u^p$, $1\le p<\infty$, we obtain $L_p$. The fundamental function $\phi_{L_F}(u)=1/F^{-1}(1/u)$, $0<u\le 1$, where $F^{-1}$ is the inverse function.

If $F$ is an Orlicz function, then the Young conjugate function $G$ is defined by
$$
G(u):=\sup_{t>0}(ut-F(t)),\;\;u>0.$$
Note that $G$ is also an Orlicz function and the Young conjugate for $G$ is $F$.

Similarly, we can define an {\it Orlicz sequence space}. Specifically, the space $\ell_{\varphi}$, where $\varphi$ is an Orlicz function, consists of all sequences $(a_{k})_{k=1}^{\infty}$ such that 
$$
\| (a_{k})_{k=1}^{\infty}\|_{\ell_{\varphi}} := \inf\left\{u>0: \sum_{k=1}^{\infty} \varphi \Big( \frac{|a_{k}|}{u} \Big)\leq 1\right\}<\infty.
$$

An Orlicz function $H$ satisfies the $\Delta_{2}^{\infty}$-condition ($H\in \Delta_{2}^{\infty}$) (resp. the $\Delta_{2}^{0}$-condition ($H\in \Delta_{2}^{0}$)) if 
$$
\limsup\limits_{t \to \infty}\frac{H(2t)}{H(t)} < \infty\;\;\mbox{(resp.}\;\;\limsup\limits_{t \to 0}\frac{H(2t)}{H(t)} < \infty).
$$
It is well known that an Orlicz function space $L_{F}$ on $[0,1]$ (resp. an Orlicz sequence space $\ell_{\varphi}$) is separable if and only if  $F\in \Delta_{2}^{\infty}$ (resp. $\varphi \in \Delta_{2}^{0}$). In this case we have  $L_F^*=L_F'=L_G$, where $G$ is the Young conjugate function for $F$ (resp. $\ell_{\varphi}^*=\ell_{\varphi}'=\ell_{\psi}$, with the Young conjugate function  $\psi$ for $\varphi$).

One can easy see (cf. \cite[Proposition~4.a.2]{LT1}) that the canonical unit vectors $e_n=(e_n^i)$, $e_n^i=\delta_{n,i}$, $n,i=1,2,\dots$, form a symmetric basis of an Orlicz sequence space $\ell_{\varphi}$ provided if $\varphi \in \Delta_{2}^{0}$. Recall that a basis $\{x_n\}_{n=1}^\infty$ of a Banach space $X$ is called {\it symmetric} if there exists $C>0$ such that for arbitrary permutation $\pi:\,\mathbb{N}\to\mathbb{N}$ and any $a_n\in\mathbb{R}$ we have
$$
C^{-1}\Big\|\sum_{n=1}^{\infty}a_nx_n\Big\|_X\le \Big\|\sum_{n=1}^{\infty}a_nx_{\pi(n)}\Big\|_X\le C\Big\|\sum_{n=1}^{\infty}a_nx_n\Big\|_X.$$

Observe that the definition of an Orlicz sequence space $\ell_{\varphi}$ depends (up to equivalence of norms) only on the behaviour of the function $\varphi$ near zero. More precisely, if $\varphi,\psi \in \Delta_{2}^{0}$, the following conditions are equivalent: (1) $\ell_{\varphi}=\ell_{\psi}$ (with equivalence of norms); 2) the unit vector bases of the spaces $\ell_{\varphi}$ and $\ell_{\psi}$ are equivalent; 3) there are $C > 0$ and $t_{0} > 0$ such that for all $0 \leq t \leq t_{0}$ it holds
$$
C^{-1}\varphi(t) \leq \psi(t) \leq C\varphi(t)
$$ 
(cf. \cite[Proposition~4.a.5]{LT1} or \cite[Theorem~3.4]{Mal}). In the case when $\varphi$ is a {\it generate} Orlicz function, i.e., there is $t> 0$ such that $\varphi(t)=0$, we have $\ell_{\varphi}=\ell_\infty$ (with equivalence of norms).

Quite similarly, the definition of an Orlicz function space $L_F$ on $[0,1]$ depends only on the behaviour of the function $F$ for large values of argument  $t$.


Let $F$ be an Orlicz function, $F\in \Delta_{2}^{\infty}$. Define the following subsets of the space $C[0, 1]$:

$$
E_{F, A}^{\infty}: = \overline{\big\{ G(x) = \frac{F(xy)}{F(y)} \ : y > A \big\}}\;(A>0), 
E_{F}^{\infty}: = \bigcap_{A > 0}E_{F, A}^{\infty}, \ \ C_{F}^{\infty}: = \overline{{\rm conv}\, E_{F}^{\infty}},
$$
where ${\rm conv}\,U$ is the convex hull of a set $U$ and the closure is taken in the space $C[0,1].$ All these sets are non-void, compact in $C[0,1]$ and consist of Orlicz functions \cite[Lemma~4.a.6]{LT1}. It is well known that the sets $E_{F}^{\infty}$ and $C_{F}^{\infty}$ rather fully determine the structure of disjoint sequences of $L_F$ (see \cite[\S\,4.a]{LT1}, \cite{LTIII}).




Let $F$ be an Orlicz function, $\varphi$ be a function defined on $[0,1]$. We will write $E_F^\infty\cong\{\varphi\}$ if for every $H\in E_F^\infty$ there is a constant $C=C(H)$ such that
\begin{equation}
\label{eq1}
C^{-1}\varphi(t)\le H(t)\le C\varphi(t),\;\;0<t\le 1.
\end{equation}
In the case when this condition is fulfilled uniformly for all $H\in E_F^\infty$, i.e., there is a constant $C>0$ such that \eqref{eq1} holds for each $H\in E_F^\infty$ we will write $E_F^\infty\equiv\{\varphi\}$. In a similar way, we understand the expressions $C_F^\infty\cong\{\varphi\}$ and $C_F^\infty\equiv\{\varphi\}$. 

Let $E_F^\infty\cong\{\varphi\}$. Suppose that $\varphi$ is a non-generate function, i.e., $\varphi(t)>0$ for $t>0$. Then $F$ is quasi-multiplicative, that is, $C^{-1}\le F(uv)/F(u)F(v)\le C$ for some $C>0$ and all $0\le u,v\le 1$ (see the proof of Theorem~4.1 in \cite{bib:04}). Therefore, by a classical result of Polya, there is $1\le p<\infty$ such that $\varphi(u)$ is equivalent on $[0,1]$ to the function $u^p$. 

It is obvious that from $E_F^\infty\equiv\{\varphi\}$ it follows that $C_F^\infty\equiv\{\varphi\}$. What is about a weaker condition $E_F^\infty\cong\{\varphi\}$? We will see further that in this case may be both situations: $C_F^\infty\cong\{\varphi\}$ and $C_F^\infty\not\cong\{\varphi\}$.

An Orlicz function $F$ such that $F\in \Delta_{2}^{\infty}$ is said to be {\it regularly varying at $\infty$} (in sense of Karamata) if the limit $\lim_{t\to\infty}F(tu)/F(t)$ exists for all $u>0$ (in fact, it suffices that the limit exists when $0<u\le 1$). Then, there is a $p$, $1\le p<\infty$, such that $\lim_{t\to\infty}F(tu)/F(t)=u^p$; in this case $F$ is called {\it regularly varying of order $p$}.

\subsection{\protect Properties of disjoint sequences in Banach lattices.}
\label{prel2}


Following to \cite{FTT-09}, a Banach lattice $X$ is said to be {\it disjointly homogeneous} (in brief, $DH$) if two arbitrary normalized disjoint sequences $\{x_n\}$ and $\{y_n\}$ from $X$ contain subsequences $\{x_{n_k}\}$ and $\{y_{n_k}\}$ that are $C$-equivalent in $X$ for some $C>0$. As usual, this means that for all $c_k\in\mathbb{R}$ we have
$$
C^{-1}\Big\|\sum_{k=1}^\infty c_kx_{n_k}\Big\|_X\le \Big\|\sum_{k=1}^\infty c_ky_{n_k}\Big\|_X\le C\Big\|\sum_{k=1}^\infty c_kx_{n_k}\Big\|_X.$$
If there exists a constant $C$, which is suitable for each pair of normalized disjoint sequences $\{x_n\}$ and $\{y_n\}$, then we will say that $X$ is {\it uniformly $DH$} \cite{FHSTT}.

Clearly, every uniformly $DH$ Banach lattice is $DH$. In the paper \cite{FHSTT}, it is shown that there are $DH$ Banach lattices ordered by basis that are not being uniformly $DH$ (a {\it Banach lattice ordered by basis} is a Banach lattice $X$ with a basis $\{x_n\}$ such that $\sum_{n=1}^\infty a_nx_n\ge 0$ if and only if $a_n\ge 0$ for all $n$; see \cite{FHSTT}).

In particular, given a $p\in[1,\infty]$, a Banach lattice $X$ is called  {\it $p$-disjointly homogeneous} (briefly, $(p-DH)$) if each  normalized disjoint sequence in $X$ has a subsequence equivalent to the unit vector basis of $\ell_p$ ($c_0$ when $p=\infty$). In the case when the constant $C$ of this equivalence can be chosen uniformly for all normalized disjoint sequences, $X$ is called uniformly $(p-DH)$. 

As was mentioned in Section~\ref{Intro}, $DH$ property is enjoyed, in particular, by $L_p(\mu)$-spaces, $1 \le p \le\infty$, Lorentz spaces $L_{q,p}$ and $\Lambda(W, p)$, certain Orlicz spaces, Tsirelson space and some others. Moreover, by using the complex method of interpolation, it was proved, in \cite{A15}, that for every $1 \le p \le\infty$ and any increasing concave function $\varphi(u)$ on $[0,1]$, which is not equivalent to neither $1$ nor $u$, there exists a $(p-DH)$ symmetric space on $[0,1]$ with the fundamental function $\varphi$.

In \cite{HST}, for $p=2$ the following weaker notion has been introduced. Let $1\le p\le\infty$. A symmetric space $X$ is said to be {\it restricted $(p-DH)$} if for every sequence of pairwise disjoint subsets $\{A_n\}_{n=1}^\infty$ of $[0,1]$ there is a subsequence $\{A_{n_k}\}$ such that $\{\frac{1}{\|\chi_{A_{n_k}}\|_X} \chi_{A_{n_k}}\}$ is equivalent to the unit vector basis of $\ell_p$ ($c_0$ when $p=\infty$). As above, we can define also {\it uniformly restricted} $(p-DH)$ symmetric spaces.

Clearly, every $(p-DH)$ symmetric space is restricted $(p-DH)$. On the other hand, for each $1\le p<\infty$, there exist restricted $(p-DH)$ symmetric spaces on $[0,1]$, which fail to be $(p-DH)$ (see \cite{A19b}). 

A symmetric space $X$ is called {\it disjointly complemented} (in brief, $DC$) if every disjoint sequence in $X$ has a subsequence whose span is complemented in $X$ (see \cite{FHSTT}).

For more detailed information related to $DH$- and $(p-DH)$-properties see \cite{FTT-09,bib:04,FHSTT,FHT-survey,HST,A15,A19a,A19b}.

The notation   $A\asymp B$ means that there exists a constant $C>0$ that does not depend on the arguments of $A$ and $B$ such that $C^{-1}{\cdot}A\le B\le C{\cdot}A$. 
Finally, throughout the paper, we set $\log a:=\log_2 a$, where $a>0$, and $\bar{\chi}_{A}:={\chi}_{A}/\|{\chi}_{A}\|.$

\section{\protect \medskip A special Orlicz function.}

The following result will play a key role in our further considerations.

\begin{propos}
\label{prop3}
There exists an Orlicz function $\Phi$ on $[0,\infty)$ such that $C_\Phi^\infty\cong\{t\}$ and $E_\Phi^\infty\not\equiv\{t\}$.
\end{propos}
\begin{proof}
Let $\{r_i\}_{i=0}^\infty$ be an increasing sequence of positive integers such that $r_0=1$ and $\lim_{i\to\infty}r_{i-1}/r_i=0$. Furthermore, for each $i=1,2,\dots$ let $\{k_i^j\}_{j=1}^\infty$ be an increasing sequence of positive integers such that the intervals $A_i^j:=[2^{k_i^j}-r_i,2^{k_i^j}]$, $i,j=1,2,\dots,$ are pairwise disjoint and are separated in the sense that size of the gaps between them tends to $\infty$. More precisely, if $\{A_k^*\}_{i=0}^\infty$ is the sequence of these sets enumerated so that $A_1^*<A_2^*<\dots$, then 
\begin{equation}
\label{eq2}
\min A_{k+1}^*-\max A_k^*\to\infty\;\;\mbox{as}\;k\to\infty.
\end{equation}

We define also the functions
$$
\varphi(s):=\sum_{i=1}^\infty (1+r_{i-1}/r_i)\sum_{j=1}^\infty\chi_{A_i^j}(s)+\chi_{[0,\infty)\setminus \cup_{i,j=1}^\infty A_i^j}(s),$$
$$
f(t):=\int_0^t\varphi(s)\,ds,$$
and
$$
F(x):=2^{f(\log x)}\chi_{(1,\infty)}(x)+x\chi_{[0,1]}(x).$$

The function $F$ need not to be convex. However, we will show that $F$ is equivalent on $[0,\infty)$ to the convex function $\Phi$ defined by
$$
\Phi(t):=\int_0^t F(s)/s\,ds.$$
To prove the convexity of $\Phi$ it suffices to check that $F(s)/s$ increases, which is equivalent to the inequality
\begin{equation}
\label{eq6}
\frac{F(st)}{F(s)}\ge t\;\;\mbox{for all}\;t\ge 1,s>0.
\end{equation}
If $s\ge 1$, then
$$
f(\log t+\log s)-f(\log s)=\int_{\log s}^{\log t+\log s}\varphi(u)\,du\ge \log t,$$
whence
$$
\frac{F(st)}{F(t)}\ge 2^{\log t}=t.$$
Suppose that $0<s<1$. Then, if $st<1$, by definition, we have $F(st)=st=tF(s)$. Otherwise, if $st\ge 1$, then
$$
f(\log (st))=\int_{0}^{\log (st)}\varphi(u)\,du\ge \log (st),$$
and therefore
$$
F(st)=2^{f(log (st))}\ge st=tF(s).$$
Thus, \eqref{eq6} is proved, which implies the convexity of $\Phi$. Moreover, from \eqref{eq6} it follows that $\Phi(t)\le F(t)$ for all $t>0$. 

Show now that
\begin{equation}
\label{eq7}
F(t)\le 4 F(t/2),\;\;t>0.
\end{equation}

Indeed, if $0<t\le 1$, we have $F(t)=t=2F(t/2)$. Since $\varphi(u)\le 2$ for $1<u\le 2$, it follows
$$
f(\log t)=\int_{0}^{\log t}\varphi(u)\,du\le 2\log t,$$
whence
$$
F(t)=2^{f(\log t)}\le t^2\le 2t=4F(t/2).$$
Finally, if $t>2$, then
$$
f(\log t)-f(\log (t/2))=\int_{\log (t/2)}^{\log t}\varphi(u)\,du\le 2.$$
Consequently,
$$
\frac{F(t)}{F(t/2)}=2^{f(\log t)-f(\log (t/2))}\le 4.$$
Summarizing these estimates, we obtain \eqref{eq7}. 

In turn, from inequalities \eqref{eq6} and \eqref{eq7} it follows that 
$$
\Phi(t)\ge \int_{t/2}^t \frac{F(u)}{u}\,du\ge F(t/2)\ge \frac14 F(t).$$
Thus, 
\begin{equation}
\label{eq7a}
\frac14 F(t)\le \Phi(t)\le F(t),\;\;t>0.
\end{equation}

Let us prove next that $E_\Phi^\infty\not\equiv\{t\}$. To this end, it suffices clearly to find functions $H_i\in E_\Phi^\infty$, $i=1,2,\dots$, such that  
\begin{equation}
\label{eq3}
2^{r_i}H_i(2^{-r_i})\to 0\;\;\mbox{as}\;i\to\infty.
\end{equation}

By the definition of $F$, for all $i,j=1,2,\dots$ 
$$
\frac{F(2^{2^{k_i^j}-r_i})}{F(2^{2^{k_i^j}})}=2^{f(2^{k_i^j}-r_i)-f(2^{k_i^j})}.$$
Therefore, since
$$
f(2^{k_i^j}-r_i)-f(2^{k_i^j})=\int_{2^{k_i^j}}^{2^{k_i^j}-r_i}\varphi(s)\,ds=-\int_{2^{k_i^j}-r_i}^{2^{k_i^j}}(1+r_{i-1}/r_i)\,ds=-r_i-r_{i-1},
$$
we have 
$$
\frac{F(2^{2^{k_i^j}-r_i})}{F(2^{2^{k_i^j}})}=2^{-r_i-r_{i-1}},\;\;i,j=1,2,\dots.$$
Then, by \eqref{eq7a}, passing to an appropriate subsequence $\{2^{2^{k_i^{j_s}}}\}_{s=1}^\infty$, for a function $H_i\in E_\Phi^\infty$ we get
$$
H_i(2^{-r_i})=\lim_{s\to\infty}\frac{\Phi(2^{2^{k_i^{j_s}}-r_i})}{\Phi(2^{2^{k_i^{j_s}}})}\asymp 2^{-r_i-r_{i-1}},\;\;i=1,2,\dots$$
with some constant independent of $i$.
Thus, \eqref{eq3} is proved and hence $E_\Phi^\infty\not\equiv\{t\}$.

It is left to show that $C_\Phi^\infty\cong\{t\}$. Since $H(t)\le t$ for every $H\in C_\Phi^\infty$, we need to prove only that $H(t)\ge ct$, $0<t\le 1$, for some constant $c>0$ depending on $H$. Observe that it suffices to check that there exists  $c>0$ such that  
\begin{equation}
\label{eq4}
H(2^{-r})\ge c 2^{-r}, r=0,1,\dots
\end{equation}
Indeed, assuming that \eqref{eq4} holds, for each $0<t\le 1$ we can find $r=0,1,\dots$ satisfying $2^{-r-1}<t\le 2^{-r}$, Then, since $H$ increases, it follows
$$
H(t)\ge H(2^{-r-1})\ge c 2^{-r-1}\ge c_1 t,$$
with $c_1:=c/2$. Moreover, one can readily see that it is sufficient to prove \eqref{eq4} only for $r\ge m$, with some $m\in\mathbb{N}$. 

By definition, every function $H\in C_\Phi^\infty$ can be represented as follows:
\begin{equation}
\label{eq5}
H(u)=\lim_{l\to\infty} y_l(u),\;\;\mbox{where}\;y_l(u):=\sum_{s=1}^{n_l}\lambda_s^l\frac{\Phi(t_s^lu)}{\Phi(t_s^l)},\;\;0<u\le 1, 
\end{equation}
where $\lambda_s^l>0$, $\sum_{s=1}^{n_l}\lambda_s^l=1$, and $\beta_l:=\min _{s=1,\dots,n_l}t_s^l\to \infty$ as $l\to\infty$.

Observe that if $y_l(2^{-r})2^{r}\ge \frac{1}{16}$ for every $r=1,2,\dots$ and all sufficiently large $l$, then we have inequality \eqref{eq4} for $c=\frac{1}{16}$, and so everything is done. Otherwise, passing to a subsequence if it is necessary, we can assume that there is $m\in\mathbb{N}$ such that for all sufficiently large $l$
\begin{equation}
\label{eq21}
y_l(2^{-m})2^{m}<\frac{1}{16}.
\end{equation}
For a fixed $m$ satisfying \eqref{eq21}, choose $i_0$ so that
\begin{equation}
\label{eq22}
\frac{r_{i-1}}{r_{i}}\le\frac{1}{m}\;\;\;\mbox{for all}\;i\ge i_0.
\end{equation}

Also, let 
$$\Delta(l,s,r):=[\log t_s^l-r,\log t_s^l],\;\; s=1,\dots,n_l,\; l,r=1,2,\dots$$
and
$$
I(l,r):=\{s=1,\dots,n_l:\,\Delta(l,s,r)\cap A_i^j\ne\emptyset\;\mbox{for some}\;1\le i\le i_0\;\mbox{and}\; j=1,2,\dots\},$$
$$
I'(l,r):=\{1,2,\dots,n_l\}\setminus I(l,r),\;\;l,r=1,2,\dots
$$
Let us estimate from above the sum $\sum_{s\in I'(l,m)}\lambda_s^l$ for sufficiently large $l=1,2,\dots$

Suppose $s\in I'(l,m)$. In the case when $\Delta(l,s,m)\cap A_i^j=\emptyset$ for all $i,j=1,2,\dots$, we have
$$
f(\log t_s^l-m)- f(\log t_s^l)=-\int_{\Delta(l,s,m)}\,du =-m,$$
whence 
$$
\frac{F(t_s^l2^{-m})}{F(t_s^l)}=2^{f(\log t_s^l-m)- f(\log t_s^l)}=2^{-m}.$$
Let now $s\in I'(l,m)$ and $\Delta(l,s,m)\cap A_i^j\ne\emptyset$ for some positive integers $i$ and $j$. Then, by definition of the set $I'(l,m)$, we deduce that $i>i_0$.
Therefore, $\varphi(u)\le 1+\max_{i\ge i_0}r_{i-1}/r_{i}$, $u\in \Delta(l,s,m)$, and from \eqref{eq22} it follows
\begin{eqnarray*}
f(\log t_s^l-m)-f(\log t_s^l)&=& -\int_{\Delta(l,s,m)}\varphi(u)\,du\ge -\int_{\Delta(l,s,m)}(1+\max_{i\ge i_0}r_{i-1}/r_{i})\,du\\
&=& -m(1+\max_{i\ge i_0}r_{i-1}/r_{i})\ge -m-1.
\end{eqnarray*}
Thus, in this case
$$
\frac{F(t_s^l2^{-m})}{F(t_s^l)}\ge 2^{-m-1}.$$
Summarizing the above estimates and appealing to inequalities \eqref{eq7a} and  \eqref{eq21}, we get  
$$
\sum_{s\in I'(l,m)}\lambda_s^l\le 2^{m+1}\sum_{s\in I'(l,m)}\lambda_s^l\frac{F(t_s^l2^{-m})}{F(t_s^l)}\le 2^{m+3}\sum_{s\in I'(l,m)}\lambda_s^l\frac{\Phi(t_s^l2^{-m})}{\Phi(t_s^l)}\le 2^{m+3} y_l(2^{-m})\le\frac12.$$
Therefore, by definition of the set $I(l,m)$, for all sufficiently large $l$ 
\begin{equation*}
\label{eq23}
\sum_{s\in I(l,m)}\lambda_s^l\ge\frac12.
\end{equation*}
Assuming that $r\ge m$, we clearly have $\Delta(l,s,m)\subset \Delta(l,s,r)$ for all  $s=1,2,\dots,n_l$ and $l=1,2,\dots$. This implies, in turn, that $I(l,m)\subset I(l,r)$, $l=1,2,\dots$.
Thus, from the last inequality it follows that for every $r>m$ and for all sufficiently large $l$
\begin{equation}
\label{eq26}
\sum_{s\in I(l,r)}\lambda_s^l\ge\frac12.
\end{equation}
Moreover, choosing $l$ sufficiently large, thanks to \eqref{eq2}, we can assume that each interval $\Delta(l,s,r)$ intersects, at most, one of the intervals $A_i^j$, $i,j=1,2,\dots$. In particular, if $s\in I(l,r)$ the interval $\Delta(l,s,r)$ intersects only a unique interval $A_i^j$ for some $1\le i\le i_0$ and $j=1,2,\dots$. Then, denoting $\alpha:=m(\Delta(l,s,r)\cap A_i^{j})$, we have $0<\alpha\le r_i$ and
\begin{eqnarray*}
f(\log t_s^l-r)-f(\log t_s^l)&=& -\int_{\Delta(l,s,r)}\varphi(u)\,du= -r+\alpha -\int_{\Delta(l,s,r)\cap A_i^{j}}(1+r_{i-1}/r_i)\,du\\ &=& -r+\alpha-\alpha(1+r_{i-1}/r_i)\ge -r-r_{i-1}\ge -r-r_{i_0-1},
\end{eqnarray*}
whence
$$
\frac{F(t_s^l2^{-r})}{F(t_s^l)}\ge 2^{-r_{i_0-1}}\cdot 2^{-r},\;\;s\in I(l,r).$$
Combining this together with inequalities \eqref{eq7a} and \eqref{eq26}, we deduce that for all sufficiently large $l$ 
$$
y_l(2^{-r})\ge \sum_{s\in I(l,r)}\lambda_s^l\frac{\Phi(t_s^l2^{-r})}{\Phi(t_s^l)}\ge \frac14\sum_{s\in I(l,r)}\lambda_s^l\frac{F(t_s^l2^{-r})}{F(t_s^l)}\ge 2^{-r_{i_0-1}-2}\cdot 2^{-r}\sum_{s\in I(l,r)}\lambda_s^l\ge 2^{-r_{i_0-1}-3}\cdot 2^{-r}.$$
Observe that the index $i_0$ in \eqref{eq22} does not depend on $r$ (it depends only on the given function $H$). Therefore, the last inequality implies estimate \eqref{eq4} for $c=2^{-r_{i_0-1}-3}$. This completes the proof of the proposition.
\end{proof}

\begin{corollary}
\label{cor3new}
For every $1<p<\infty$ there exists an Orlicz function $\Phi_p$ on $[0,\infty)$ such that $C_{\Phi_p}^\infty\cong\{t^p\}$ and $E_{\Phi_p}^\infty\not\equiv\{t^p\}$.
\end{corollary}
\begin{proof}
Let $\Phi$ be the Orlicz function from Proposition~\ref{prop3}. We set $\Phi_p(u):=\Phi(u^p)$, $0\le u\le 1$. Clearly, $\Phi_p$ is an Orlicz function. Moreover, one can readily check that
$$
H\in E_{\Phi_p}^\infty\;\;\mbox{if and only if}\;\;H(u^{1/p})\in E_{\Phi}^\infty$$
and
$$
H\in C_{\Phi_p}^\infty\;\;\mbox{if and only if}\;\;H(u^{1/p})\in C_{\Phi}^\infty.$$
Indeed, assuming that $H\in E_{\Phi_p}^\infty$, we can find a sequence $\{t_n\}$ such that $t_n\uparrow\infty$ as $n\to\infty$ and
$$
H(u)=\lim_{n\to\infty}\frac{\Phi_p(ut_n)}{\Phi_p(t_n)}=\lim_{n\to\infty}\frac{\Phi(u^pt_n^p)}{\Phi(t_n^p)}.$$
Since $t_n^p\uparrow\infty$, the function $H(u^{1/p})$ belongs to the set $E_{\Phi}^\infty$. The converse can be shown in the same way.
Since the result for the set $C_{\Phi_p}^\infty$ can be obtained similarly, the desired result follows now easily from Proposition~\ref{prop3}. 
\end{proof}

\section{\protect \medskip $(p-DH)$ Orlicz spaces, $1\le p<\infty$.}

We start this section with proving necessary and sufficient conditions, under which an Orlicz space is restricted $(p-DH)$. We provide the proof of this simple result for the reader's convenience and as well as to track the constants in the inequalities for future purposes.


Recall that $\bar{\chi}_{A}:={\chi}_{A}/\|{\chi}_{A}\|_{L_F}.$

\begin{propos}
\label{prop1}
Let $1\le p<\infty$, and let $F$ be an Orlicz function. 
The following conditions are equivalent:

(a) $L_F$ is restricted $(p-DH)$;

(b) for every sequence of disjoint sets $E_n\subset [0,1]$, $n=1,2,\dots$, there exists a subsequence $E_{n_k}$, $k=1,2,\dots$, such that 
\begin{equation}
\label{eq1new}
C^{-1}m^{1/p}\le\Big\|\sum_{k=1}^m \bar{\chi}_{E_{n_k}}\Big\|_{L_F}\le Cm^{1/p},\;\;m=1,2,\dots
\end{equation}
with some constant $C>0$;

(c) $E_F^\infty\cong\{t^p\}$.
\end{propos}
\begin{proof}
Let us show first that for every $H\in E_F^\infty$ and any $\varepsilon>0$ there is a sequence of disjoint sets $E_n\subset [0,1]$, $n=1,2,\dots$, such that for all $(c_n)\in \ell_H$
\begin{equation}
\label{equa2}
(1-\varepsilon)\|(c_n)\|_{\ell_H}\le \Big\|\sum_{n=1}^\infty c_n\bar{\chi}_{E_n}\Big\|_{L_F}\le (1+\varepsilon)\|(c_n)\|_{\ell_H}.
\end{equation}

Since $H\in E_F^\infty$, then there is a sequence $\{t_n\}_{n=1}^\infty$ such that $t_n\to\infty$ and
$$
\left|\frac{F(t_nu)}{F(t_n)}-H(u)\right|<\frac{\varepsilon}{2^n}$$
for all $0\le u\le 1$ and $n=1,2,\dots$
This implies that for all $c_n\in\mathbb{R}$
\begin{equation}
\label{equa3}
\sum_{n=1}^\infty H(c_n)-\varepsilon\le \sum_{n=1}^\infty \frac{F(t_nc_n)}{F(t_n)}\le \sum_{n=1}^\infty H(c_n)+\varepsilon.
\end{equation}
Without loss of generality, we may assume that $F(t_n)\ge 2^n$, $n=1,2,\dots$. Thanks to that, there are disjoint sets $E_n\subset [0,1]$ such that $m(E_n)=1/F(t_n)$, $n=1,2,\dots$. Denoting $f_n:=\bar{\chi}_{E_n}$, suppose that $\|\sum_{n=1}^\infty c_nf_n\|_{L_F}\le 1$. Since $\|\chi_{E_n}\|_{L_F}=1/F^{-1}(1/m(E_n))=1/t_n$ (see Section~\ref{Orlicz}), it follows
$$
\sum_{n=1}^\infty \frac{F(t_n|c_n|)}{F(t_n)}=\int_0^1F\Big(\Big|\sum_{n=1}^\infty c_nf_n(u)\Big|\Big)\,du\le 1.$$
By \eqref{equa3}, this yields
$$
\sum_{n=1}^\infty H(|c_n|)\le 1+\varepsilon,$$
and, as $H$ is a convex function, $\|(c_n)\|_{\ell_H}\le 1+\varepsilon$. 

Conversely, if $\|(c_n)\|_{\ell_H}\le 1$, we have $\sum_{n=1}^\infty H(|c_n|)\le 1$. Therefore, according to \eqref{equa3},  
$$
\int_0^1F\Big(\Big|\sum_{n=1}^\infty c_nf_n(u)\Big|\Big)\,du\le 1+\varepsilon.$$
Hence, $\|\sum_{n=1}^\infty c_nf_n\|_{L_F}\le 1+\varepsilon.$
In view of the obtained estimates, applying the simple homogeneity argument, we come to inequality \eqref{equa2}.

Proceeding with the proof of the proposition, note that the implication $(a)\Rightarrow (b)$ is obvious. Assuming now that $(b)$ holds, we prove $(c)$.

Let $H\in E_F^\infty$ be arbitrary. Then, setting $\varepsilon=\frac12$, we can find  disjoint sets $E_n\subset [0,1]$, $n=1,2,\dots$, such that for all $(c_n)\in \ell_H$ we have \eqref{equa2}. In particular, for all $m=1,2,\dots$ it follows that
$$
\frac12\Big\|\sum_{n=1}^m e_n\Big\|_{\ell_H}\le \Big\|\sum_{n=1}^m \bar{\chi}_{E_n}\Big\|_{L_F}\le 2\Big\|\sum_{n=1}^m e_n\Big\|_{\ell_H},$$
where $e_n$ are the vectors of the unit basis of $\ell_H$. On the other hand, according the hypothesis, passing to a subsequence (and preserving the notation), we get  
\begin{equation*}
\label{eq1new}
C^{-1}m^{1/p}\le\Big\|\sum_{n=1}^m \bar{\chi}_{E_{n}}\Big\|_{L_F}\le Cm^{1/p},\;\;m=1,2,\dots
\end{equation*}
with some constant $C>0$. In consequence,
$$
(2C)^{-1}m^{-1/p}\le H^{-1}(1/m)\le(2C)^{-1}m^{-1/p},\;\;m=1,2,\dots,$$ 
and so $H(u)\asymp u^{1/p}$, $0<u\le 1$, with a constant that depends only on $C$. Thus, $E_F^\infty\cong\{t^p\}$, which completes the proof of $(c)$.

Finally, we will prove the implication $(c)\Rightarrow (a)$. Let $E_n\subset [0,1]$, $n=1,2,\dots$, be disjoint, $f_n:=\bar{\chi}_{E_n}$. We introduce the functions
$$
H_n(u):= \int_0^1F(uf_n(v))\,dv=F(uF^{-1}(1/m(E_n)))m(E_n).$$
Clearly,
$$
H_n(u)= \frac{F(ut_n)}{F(t_n)},\;\;\mbox{where}\;t_n:=F^{-1}(1/m(E_n)),\;n=1,2,\dots$$
Since $m(E_n)\to 0$, then $t_n\to\infty$ as $n\to\infty$. Therefore, since $\{H_n\}$ is a relatively compact set in $C[0,1]$ \cite[Lemma~4.a.6 and  subsequent Remark]{LT1}, there is a subsequence $\{H_{n_j}\}$, uniformly converging on $[0,1]$ to some $H\in E_F^\infty$. We may assume that $|H_{n_j}(u)-H(u)|\le 2^{-j}$ for all $0\le u\le 1$ and $j=1,2,\dots$, whence
$$
\sum_{j=1}^\infty H(|c_j|)-1\le \sum_{j=1}^\infty H_{n_j}(|c_j|)\le \sum_{j=1}^\infty H(|c_j|)+1$$
for all $c_j\in\mathbb{R}$. Since $E_F^\infty\cong\{t^p\}$, then $C^{-1}u^p\le H(u)\le Cu^p$ for some $C\ge 1$ and all $0\le u\le 1$. Consequently,
$$
C^{-1}\sum_{j=1}^\infty |c_j|^p-1\le \sum_{j=1}^\infty H_{n_j}(|c_j|)\le C\sum_{j=1}^\infty |c_j|^p+1,$$
or equivalently
$$
C^{-1}\sum_{j=1}^\infty |c_j|^p-1\le \int_0^1F\Big(\Big|\sum_{j=1}^\infty c_jf_{n_j}(v)\Big|\Big)\,dv\le C\sum_{j=1}^\infty |c_j|^p+1.$$
Hence, if $\|(c_j)\|_{\ell_p}\le 1$, then $\|\sum_{j=1}^\infty c_jf_{n_j}\|_{L_F}\le C+1$. Conversely, if  $\|\sum_{j=1}^\infty c_jf_{n_j}\|_{L_F}\le 1$, then $\|(c_j)\|_{\ell_p}\le (2C)^{1/p}$. Thus, the subsequence $\{f_{n_j}\}$ is $2C$-equivalent in $L_F$ to the unit vector basis of $\ell_p$, and the proof is completed.

\end{proof}

To prove a similar criterion for $DH$-property, we will need the next useful result, which is well known in the separable case (see \cite[Proposition~3]{LTIII}).


\begin{lemma}
\label{le2}
Let $F$ be an Orlicz function. Then, every normalized disjoint sequence from the Orlicz space $L_F$ contains a subsequence that is $6$-equivalent to the unit vector basis of an Orlicz sequence space $\ell_H$ with some $H\in C_F^\infty$.
\end{lemma}
\begin{proof}
Let $\{f_n\}_{n=1}^\infty\subset L_F$ be a disjoint normalized sequence. 
We claim that the lemma will be proved once we select a subsequence $\{ f_{n_{k}} \}\subset\{ f_{n} \}$ such that $f_{n_{k}} = u_k+v_k,$ where the functions $u_k$ and $v_k$, $k=1,2,\dots$, satisfy the conditions:

(a) $u_k\cdot v_i=0$ for all $k,i=1,2,\dots$;

(b) the sequence $\{ u_k\}_{k=1}^{\infty}$ $4$-equivalent to the unit vector basis of an Orlicz sequence space $l_{H}$ with some $H \in C_{F}^{\infty}$, i.e., for all $(a_k)_{k=1}^\infty\in l_H$
\begin{equation}\label{14012020}
\frac14\| (a_{k}) \|_{l_{H}} \leq \Big\| \sum_{k=1}^{\infty}a_{k}u_{k}\Big\|_{L_F} \leq 4\| (a_{k}) \|_{l_{H}}.
\end{equation}

(c) for each sequence $(a_k)_{k=1}^\infty\in c_0$ we have
\begin{equation}\label{120120206}
\Big\| \sum_{k=1}^{\infty}a_{k}v_{k}\Big\|_{L_F} \leq 2\| (a_{k})\|_{c_{0}}.
\end{equation}

Indeed, on the one hand, from $\eqref{14012020}$ and $\eqref{120120206}$ it follows that
\begin{eqnarray*}
\Big\| \sum_{k=1}^{\infty}a_{k}f_{n_{k}} \Big\|_{L_F} 
&\leq&
\Big\| \sum_{k=1}^{\infty}a_{k}u_{k}\Big\|_{L_F} +\Big\| \sum_{k=1}^{\infty}a_{k}v_{k} \Big\|_{L_F} 
\\ &\leq& 4\| (a_{k}) \|_{l_{H}}+ {2}\|(a_{k})\|_{c_{0}}\\ &\leq &
6\| (a_{k})\|_{l_{H}}.
\end{eqnarray*}
On the other hand, applying $(a)$ and once more $\eqref{14012020}$, we obtain the opposite inequality:
$$
\Big\| \sum_{k=1}^{\infty}a_{k}f_{n_{k}}\Big\|_{L_F} \geq \Big\| \sum_{k=1}^{\infty}a_{k}u_{k}\Big\|_{L_F} \geq \frac14\|(a_{k})\|_{l_{H}}.
$$
Thus, our claim is proved. Consequently, it suffices to show that such a subsequence $\{f_{n_{k}}\}_{k=1}^{\infty}$ exists. 

Since $\| f_{n}\|_{L_F}= 1$, by definition of the Luxemburg-Nakano norm, for all  $n = 1, 2, \dots$ we have
$$
\int\limits_{0}^{1}F\Big(\frac{|f_{n}(t)|}{2}\Big)\,dt \leq 1\;\;\mbox{and}\;\;\int\limits_{0}^{1}F(2|f_{n}(t)|)\,dt > 1.
$$
Hence, from absolute continuity of Lebesgue integral it follows that for each $n = 1, 2, \dots$ there is a constant $M_{n}>0$ such that for the functions $x_n:=f_n\chi_{\{|f_n|>M_n\}}$ and $y_n:=f_n\chi_{\{|f_n|\le M_n\}}$ we have

\begin{equation}\label{18122019}
\int\limits_{0}^{1}F\Big(\frac{|x_{n}(t)|}{2}\Big)\,dt \leq 2^{-n}\;\; \text{and}\;\; \int\limits_{0}^{1}F(2|y_{n}(t)|)\,dt > 1.
\end{equation}
Observe that the functions $y_n$, $n = 1,2, \dots$, are disjoint and $y_n\in L_\infty\subset L_{F}^0$, where $L_{F}^0$ is the separable part of the Orlicz space $L_F$ (see Section~\ref{Orlicz}). Moreover, $\| y_{n} \|_{L_F} \leq \| f_{n} \|_{L_F} \leq 1$ and in view of the second inequality in \eqref{18122019} we have $\| y_{n} \|_{L_F} > {1}/{2}$ for all $n = 1,2, \dots$. Thanks to these properties, reasoning precisely as in the proof of \cite[Proposition~3]{LTIII}, we can find a subsequence $\{ y_{n_{k}} \}\subset\{ y_{n} \}$, which is $4$-equivalent to the unit vector basis of an Orlicz sequence space $l_{H},$ where $H \in C_{F}^{\infty}.$ Therefore, the functions $u_k:=y_{n_{k}}$, $v_k:=x_{n_{k}}$, $k = 1,2, \dots$, satisfy the conditions $(a)$, $(b)$ and $x_{n_{k}}=u_k+v_k$, $k = 1,2, \dots$. 

Moreover, let $\| (a_{k}) \|_{c_{0}} \leq 1$. Since $x_{n_{k}}$, $k=1,2,\dots,$ are disjoint, from the first inequality in \eqref{18122019} it follows
\begin{eqnarray*}
\int\limits_{0}^{1} F\Big( \frac{\big(\big|\sum_{k=1}^{\infty}a_{k}v_{k}(t)\big|\big)}{2} \Big)\,dt &\leq&\int\limits_{0}^{1} F\Big( \frac{\Big(\big|\sum_{k=1}^{\infty}v_{k}(t)\big|\Big)}{2} \Big)\,dt\\ 
&=&\sum_{k=1}^{\infty} \int\limits_{0}^{1}F\Big( \frac{|x_{n_{k}}(t)|}{2} \Big)\,dt \leq \sum_{k=1}^{\infty}2^{-n_{k}} \leq 1.
\end{eqnarray*}
Applying now the homogeneity argument precisely in the same way as above, we get inequality $\eqref{120120206}$. Thus, the proof is completed. 
\end{proof}

\begin{rem}
\label{rem1}
One can readily see that in the case when $\{f_{n}\}$ is a normalized sequence of characteristic functions of disjoint subsets of $[0,1]$ the function $H$ from the proof of Lemma~\ref{le2} belongs to the smaller set $E_F^\infty$.
\end{rem}

\begin{theorem}
\label{th3new}
Let $1\le p<\infty$, and let $F$ be an Orlicz function. 
Then, the Orlicz space $L_F$ is $(p-DH)$ if and only if $C_F^\infty\cong\{t^p\}$.
\end{theorem}
\begin{proof}
Assume first that $L_F$ is a $(p-DH)$ space. Let $H\in C_F^\infty$. Then, arguing similarly as in the beginning of the proof of Proposition~\ref{prop1}, we can find a sequence of normalized disjoint functions $\{f_n\}$ equivalent to the unit vector basis of the sequence Orlicz space $\ell_H$. On the other hand, by the hypothesis, some subsequence $\{f_{n_k}\}\subset \{f_n\}$ is equivalent in $L_F$ to the unit vector basis of $\ell_p$. 
Since the unit vector basis of an arbitrary sequence Orlicz space is a symmetric basic sequence, we conclude that $H(u)\asymp u^p$, $0\le u\le 1$. Therefore, $C_{F}^\infty\cong\{t^p\}$.

Conversely, let $C_{F}^\infty\cong\{t^p\}$. By Lemma~\ref{le2}, every normalized disjoint sequence from $L_F$ contains a subsequence that is $6$-equivalent to the unit vector basis of an Orlicz sequence space $\ell_H$ for some $H\in C_F^\infty$. Since $H(u)\asymp u^p$, $0\le u\le 1$, then obviously this subsequence is equivalent in $L_F$ also to the unit vector basis of $\ell_p$. As a result, we conclude that $L_F$ is $(p-DH)$, and the proof is completed.
\end{proof}

\begin{theorem}
\label{prop2}
Let $1\le p<\infty$, and let $F$ be an Orlicz function. 
The following conditions are equivalent:

(a) $L_F$ is uniformly $(p-DH)$;

(b) $L_F$ is uniformly restricted $(p-DH)$;

(c) there exists a constant $C>0$ such that for every sequence of disjoint sets $E_n\subset [0,1]$, $n=1,2,\dots$, we have
\begin{equation}
\label{eq1new}
C^{-1}m^{1/p}\le\Big\|\sum_{n=1}^m \bar{\chi}_{E_n}\Big\|_{L_F}\le Cm^{1/p},\;\;m=1,2,\dots;
\end{equation}

(d) $E_F^\infty\equiv\{t^p\}$;

(e) $C_F^\infty\equiv\{t^p\}$.
\end{theorem}
\begin{proof}
The implication $(a)\Rightarrow (b)$ is obvious. Next, an inspection of constants appearing in the proof of Proposition \ref{prop1} yields immediately the equivalence $(b)\Leftrightarrow (c)\Leftrightarrow (d)$. The implication $(d)\Rightarrow (e)$ follows easily from definition of the sets 
$E_F^\infty$ and $C_F^\infty$. Finally, appealing to the proof of Theorem \ref{th3new} and taking into account that the constant of equivalence of a given sequence of normalized disjoint functions $\{f_n\}\subset L_F$ to the unit vector basis of $\ell_p$ can be chosen uniformly for all such sequences, we see that $(e)$ implies $(a)$.  
\end{proof}

In the next section, we prove similar results for the (uniform) $(\infty-DH)$-property.

\section{\protect \medskip $(\infty-DH)$ Orlicz spaces.}

Here, we will make use of results by Knaust and Odell about normalized weakly null sequences dominated by the unit vector basis of $c_0$ or $\ell_p$ from the papers \cite{KO1} and \cite{KO2} respectively (more general theorems of such a sort see in \cite{Freeman}). So, it will be convenient to adopt the terminology similar to that used in these papers.

Let $1<p\le\infty$. We say that a Banach lattice $X$ has property $(D_p)$ if every disjoint sequence $\{x_n\}$, $\|x_n\|_X\le 1$, $n=1,2,\dots$, in $X$ admits a subsequence $\{x_{n_k}\}$, which is $C$-dominated, for some $C>0$, by the unit vector basis of $\ell_p$, i.e.,
$$
\Big\|\sum_{k=1}^\infty c_kx_{n_k}\Big\|_X\le C\|(c_k)\|_{\ell_p}$$
for all $(c_k)\in \ell_p$ (as above, by $\ell_\infty$ we mean $c_0$). 
Moreover, $X$ has property $(UD_p)$, if the constant $C$ can be chosen uniformly for all disjoint sequences $\{x_n\}$ in $X$ such that $\|x_n\|_X\le 1$, $n=1,2,\dots$.

Let $1<p\le\infty$ and let $X$ be a Banach lattice. A disjoint sequence $\{x_n\}$  in $X$ is called {\it $u\ell_p$-sequence} (resp. a {\it $C$-$u\ell_p$-sequence}) if $\|x_n\|_X\le 1$ for all $n=1,2,\dots$ and $\{x_n\}$ satisfies an upper $\ell_p$-estimate (resp. the $C$-upper $\ell_p$-estimate). 
We say that a sequence $\{x_n\}$ in $X$ is an $M$-{\it bad} $u\ell_p$-sequence for a constant $M <\infty$ if $\{x_n\}$ is an $u\ell_p$-sequence, and no subsequence of $\{x_n\}$ is an $M$-$u\ell_p$-sequence. An array $\{x_i^n\}_{n,i=1}^\infty$  of elements in $X$ is called a {\it bad} $u\ell_p$-array, if each column $\{x_i^n\}_{i}$  is an $M_n$-{\it bad} $u\ell_p$-sequence for all $n\in\mathbb{N}$ and $M_n\to\infty$ as $n\to\infty$. An array $\{y_i^n\}_{n,i=1}^\infty$ is called a {\it subarray} of an array $\{x_i^n\}_{n,i=1}^\infty$ whenever each column of $\{y_i^n\}_{n,i}$ is a subsequence of $\{x_i^{k_n}\}_{i}$ for some sequence $k_1 < k_2 <\dots$.  Finally, let us say that a bad $u\ell_p$-array $\{x_i^n\}_{n,i=1}^\infty$ satisfies the $\ell_p$-{\it array procedure} if there exists a subarray $\{y_i^n\}$ of $\{x_i^n\}_{n,i}$ and there exist  $a_n>0$ with $\sum_{n=1}^\infty a_n\le 1$ so that the elements $y_i=\sum_{n=1}^\infty a_ny_i^n$, $i=1,2,\dots$ are disjoint and form a sequence having no $u\ell_p$-subsequence.

\begin{propos}
\label{prop5}
Let $1<p\le\infty$. A symmetric space $X$ on $[0,1]$ with property $(D_p)$ has property $(UD_p)$.
\end{propos}
\begin{proof}
On the contrary, suppose that $X$ fails to have property $(UD_p)$. Then, it is plain that in $X$ there is a bad $u\ell_p$-array $\{x_i^n\}_{n,i=1}^\infty$. Since the space $X$ is symmetric, without loss of generality, we may assume that all the functions $x_i^n$, $n,i=1,2,\dots$, are pairwise disjoint. Moreover, observe that each column $\{x_i^n\}_{i=1}^\infty$  is a disjoint sequence containing a subsequence admitting an upper $\ell_p$-estimate with $p>1$. Hence, by Rosenthal's $\ell_1$ theorem (see e.g. \cite[Theorem~10.2.1]{AK}), it may be assumed also that the sequence $\{x_i^n\}_{i}$ is weakly null for each $n$. Then, by \cite[Theorem~2]{KO2}, if $1<p<\infty$ and, by \cite[Theorem~3.3]{KO1}, if $p=\infty$ we conclude that this array satisfies the $\ell_p$-array procedure. Thus, there are  $a_n>0$, $\sum_{i=1}^\infty a_n\le 1$, an increasing sequence of positive integers $\{k_n\}_{n=1}^\infty$, and subsequences $\{y_i^n\}_{i=1}^\infty$ of the sequences $\{x_i^{k_n}\}_{i=1}^\infty$, $n=1,2,\dots$, such that the sequence $y_i:=\sum_{n=1}^\infty a_ny_i^n$, $i=1,2,\dots$, has no subsequence admitting an upper $\ell_p$-estimate. Observe that the elements $y_i$, $i=1,2,\dots$, are pairwise disjoint and $\|y_i\|_X\le 1$.  Since this contradicts the hypothesis that $X$ possesses property $(D_p)$, the proof is completed.
\end{proof}

\begin{corollary}
\label{cor1new}
Every $(\infty-DH)$ symmetric space is uniformly $(\infty-DH)$.
\end{corollary}
\begin{proof}
Since for any disjoint normalized sequence $\{x_n\}_{n=1}^\infty$ from a Banach lattice $X$ and all $c_n\in\mathbb{R}$ we have
$$
\Big\|\sum_{n=1}^\infty c_nx_n\Big\|_X\ge \|(c_n)\|_{c_0},$$
then it suffices to apply Proposition~\ref{prop5}.   
\end{proof}


\begin{corollary}
\label{cor2new}
For every Orlicz function $F$ the following conditions are equivalent:

(i) $L_F$ is uniformly $(\infty-DH)$;

(ii) $L_F$ is $(\infty-DH)$;

(iii) each function from the set $C_F^\infty$ is generate.
\end{corollary}
\begin{proof}
Note that a function $H\in C_F^\infty$ is generate if and only if the unit vector basis in the Orlicz space $l_H$ spans $c_0$. Therefore, the equivalence $(ii)\Leftrightarrow (iii)$ can be obtained, by using Proposition~\ref{prop1} and Lemma~\ref{le2}, in the same way as a similar result for finite $p$ in Theorem~\ref{th3new}. Thus, the result follows from Corollary~\ref{cor1new}.
\end{proof}

\section{\protect \medskip Uniform DH-property and duality of DH-property in the class of Orlicz spaces.}

We start with duality results related to the non-reflexive case.
The following lemma establishes a useful link between the uniform restricted $(1-DH)$-property of Orlicz spaces and the classical notion of a regularly varying function of order $1$ at $\infty$.

\begin{lemma}
\label{le1}
Let $F$ be an Orlicz function. Then, $E_F^\infty\equiv\{t\}$ if and only if $F$ is equivalent to an Orlicz function that is regularly varying  of order $1$ at $\infty$.
\end{lemma}
\begin{proof}
Assume first that $F$ is equivalent to a regularly varying Orlicz function of order $1$ at $\infty$. Then, by \cite[Lemma~1.6]{Kal}, there exists a constant $c>0$ such that for any $0<u_0\le 1$ there is $t_0>0$ so for all $t\ge t_0$ and $u\in [u_0,1]$ we have
$$
F(ut)\ge cuF(t).$$
On the other hand, by definition, for an arbitrary $H\in E_F^\infty$, there is a sequence $t_n\uparrow\infty$ such that
$$
H(u):=\lim_{n\to\infty}\frac{F(ut_n)}{F(t_n)}.$$
Combining this with the preceding inequality, we conclude that $H(u)\ge cu$ for all $u\in [u_0,1]$. Since the constant $c$ does not depend on $u_0$, the last estimate may be extended to the whole interval $[0,1]$. Taking into account that $H(u)\le u$, $0\le u\le 1$, we get $E_F^\infty\equiv\{t\}$.

To prove the converse, we assume, on the contrary, that $F$ is not equivalent to any regularly varying Orlicz function of order $1$ at $\infty$. Applying once more \cite[Lemma~1.6]{Kal}, we get then that for arbitrary $C>0$ there exists $0<u_0\le 1$ such that for every $t_0>0$ we can find $t\ge t_0$ and $u\in [u_0,1]$ satisfying the inequality:
$$
F(ut)\le\frac{1}{C}uF(t).$$
This assumption implies that for each $n\in\mathbb{N}$ there are $0<u_0^n\le 1$, a sequence $\{t_k^n\}$, $\lim_{k\to\infty}t_k^n=\infty$, and $u_k^n\in [u_0^n,1]$ such that 
\begin{equation}
\label{equa4new}
F(u_k^nt_k^n)\le\frac{1}{n}u_k^nF(t_k^n),\;\;n,k=1,2,\dots
\end{equation}
Passing to subsequences for every $n\in\mathbb{N}$ (without changing the notation), we can assume that $u_k^n\to u_n\in [u_0^n,1]$ as $k\to\infty$ and 
$$
H_n(u):=\lim_{k\to\infty}\frac{F(ut_k^n)}{F(t_k^n)}\in E_F^\infty$$ 
(with uniform convergence on $[0,1]$). 

Fix $n\in\mathbb{N}$. Then, since $u_n>0$, there is a positive integer $k_0$ (depending on $n$) such that for all $k\ge k_0$ and $0\le u\le 1$ we have 
$$
H_n(u)\le \frac{F(ut_k^n)}{F(t_k^n)}+\frac1n u_n.$$
Inserting in this inequality $u=u_k^n$, $k\ge k_0$, and using estimate \eqref{equa4new}, we obtain
$$
H_n(u_k^n)\le \frac{F(u_k^nt_k^n)}{F(t_k^n)}+\frac1n u_n\le \frac1n u_k^n+\frac1n u_n,\;\;k\ge k_0.$$
Observe that, by continuity, $H_n(u_k^n)\to H(u_n)$ as $k\to\infty$ for each $n$. Hence, taking in the last inequality the limit as $k\to\infty$ yields 
$$
H_n(u_n)\le \frac2n u_n,\;\;n=1,2,\dots$$
Since $H_n\in E_F^\infty$, clearly, the last inequality implies that $E_F^\infty\not\equiv\{t\}$. This contradiction completes the proof of the lemma.
\end{proof}

It is known that if $X$ is an $(\infty-DH)$ Banach lattice, then the dual $X^*$ is $(1-DH)$ \cite[Theorem~2.2]{FTT-09}. In the same paper it was showed that the converse result, in general, is not true. In particular, the Lorentz space $L_{p,1}[0,1]$, $1<p<\infty$, is $(1-DH)$ but its dual $L_{q,\infty}[0,1]$, $1/p+1/q=1$, fails to be $(DH)$, because it contains a disjoint sequence equivalent to the unit vector basis of $\ell_p$. Moreover, we establish here lack of duality for $(1-DH)$-property even inside the class of Orlicz spaces. However, we begin with a positive result for the uniform $(1-DH)$ and $(\infty-DH)$-properties. 

\begin{theorem}
\label{th1}
Let $F$ and $G$ be mutually Young conjugate Orlicz functions. Then the following conditions are equivalent:

(i) there exists a constant $C_0>0$ such that
\begin{equation}
\label{eq8}
\lim_{t\to\infty}\frac{G(C_0t)}{G(t)}=\infty.
\end{equation}

(ii) $L_G$ is uniformly $(\infty-DH)$;

(iii) $L_F$ is uniformly $(1-DH)$;

(iv) $L_F$ is uniformly restricted $(1-DH)$;

(v) $F$ is equivalent to a regularly varying Orlicz function of order $1$ at $\infty$;

(vi) $E_F^\infty\equiv\{t\}$.

(vii) $L_G$ is uniformly restricted $(\infty-DH)$;

\end{theorem}
\begin{proof}
$(i)\Rightarrow (ii)$.
Let $\{g_n\}_{n=1}^\infty$ be a normalized disjoint sequence in $L_G$. Reasoning similarly as in \cite[Theorem~2.8]{Al1994}, we show that 
\begin{equation}
\label{eq9}
\lim_{n\to\infty}\int_0^1 G(|g_n(t)|/(2C_0))\,dt=0.
\end{equation}

By \eqref{eq8}, $\lim_{t\to\infty}{G(t/C_0)}/{G(t)}=0$. Therefore, given an $\varepsilon>0$ there is $t_0>0$ such that for all $t\ge t_0$
$$
\frac{G(t/(2C_0))}{G(t/2)}<\frac{\varepsilon}{2}.$$
Furthermore, denoting $E_n:={\rm supp}\,g_n$, $n=1,,2,\dots$, we have $m(E_n)\to 0$ as $n\to\infty$, and hence there exists $N\in\mathbb{N}$ such that for $n\ge N$
$$
m(E_n)\le \frac{\varepsilon}{2G(t_0/C_0)}.$$
Since $\|g_n\|_{L_G}=1$, then $\int_0^1 G(|g_n(t)|/2)\,dt\le 1$, $n=1,2,\dots$. Therefore, combining the above inequalities, we get for all $n\ge N$
\begin{eqnarray*}
\int_0^1 G(|g_n(t)|/(2C_0))\,dt &=& \int_{\{|g_n|<t_0\}\cap E_n} G(|g_n(t)|/(2C_0))\,dt + \int_{\{|g_n|\ge t_0\}} G(|g_n(t)|/(2C_0))\,dt\\ &\le& 
G(t_0/C_0)m(E_n)+\frac{\varepsilon}{2}\int_{0}^1 G(|g_n(t)|/2)\,dt\le\varepsilon,
\end{eqnarray*}
and \eqref{eq9} is established.

From \eqref{eq9} it follows the existence of a subsequence $\{g_{n_k}\}\subset \{g_n\}$ satisfying the condition:
$$
\int_0^1 G\Big(\frac{|\sum_{k=1}^\infty g_{n_k}(t)|}{2C_0}\Big)\,dt=
\sum_{k=1}^\infty \int_0^1 G\Big(\frac{|g_{n_k}(t)|}{2C_0}\Big)\,dt\le 1.$$
Hence, $\|\sum_{k=1}^\infty g_{n_k}\|_{L_G}\le 2C_0$. Summarizing, we see that for every sequence $(b_k)\in c_0$  
$$
\|(b_k)\|_{c_0}\le \Big\|\sum_{k=1}^\infty b_kg_{n_k}\Big\|_{L_G}\le 2C_0\|(b_k)\|_{c_0}.$$
Thus, $\{g_n\}$ contains a subsequence $\{g_{n_k}\}$ equivalent to the unit vector basis of $c_0$ with a constant independent of $\{g_n\}$. This means that (ii) is proved.

$(ii)\Rightarrow (iii)$.
Let $\{f_n\}_{n=1}^\infty$ be a normalized disjoint sequence from $L_F$. Note that for any Orlicz function $H$ and all non-negative $h_n,h\in L_H$, $n=1,2,\dots$, from the well-known properties of integral, the condition $h_n\uparrow h$ a.e. implies that $\|h_n\|_{L_H}\to  \|h\|_{L_H}$ as $n\to\infty$. Therefore, by \cite[Proposition~1.b.18]{LT-79}, the K\"{o}the dual $L_F'=L_G$ is a norming subspace of $L_F^*$. Consequently, we can find a disjoint sequence $\{g_n\}_{n=1}^\infty$ from $L_G$, $\|g_n\|_{L_G}=1$, such that $\int_0^1f_n(t)g_m(t)\,dt=\delta_{n,m}$, $n,m=1,2,\dots$. Then, by condition, there is a subsequence $\{g_{n_k}\}\subset\{g_n\}$ such that
$$
\Big\|\sum_{k=1}^\infty b_kg_{n_k}\Big\|_{L_G}\le C\|(b_k)\|_{c_0}$$
for a uniform constant $C>0$ and all $(b_k)\in c_0$. Hence, for any $m\in\mathbb{N}$ and all $a_k\in\mathbb{R}$ 
\begin{eqnarray*}
\Big\|\sum_{k=1}^m a_kf_{n_k}\Big\|_{L_F} &=& \sup\Big\{\int_0^1\Big(\sum_{k=1}^m |a_k|f_{n_k}(t)\Big)g(t)\,dt:\,g\in L_G, \|g\|_{L_G}\le 1\Big\}\\
&\ge & C^{-1}\int_0^1\Big(\sum_{k=1}^m |a_k|f_{n_k}(t)\Big)\Big(\sum_{k=1}^m g_{n_k}(t)\Big)\,dt\\
&=& C^{-1}\sum_{k=1}^m |a_k|\int_0^1 f_{n_k}(t)g_{n_k}(t)\,dt\ge C^{-1}\sum_{k=1}^m |a_k|.
\end{eqnarray*}
As a result, for every $m\in\mathbb{N}$ and all $a_k\in\mathbb{R}$, $k=1,2,\dots,m$ it follows that
$$
C^{-1}\sum_{k=1}^m |a_k|\le \Big\|\sum_{k=1}^m a_kf_{n_k}\Big\|_{L_F}\le \sum_{k=1}^m |a_k|.$$
Since the constant $C$ does not depend on a normalized disjoint sequence $\{f_n\}_{n=1}^\infty\subset L_F$ and $m\in\mathbb{N}$ is arbitrary, we get the desired result.

The equivalence of conditions $(iii)$, $(iv)$, $(v)$, and $(vi)$ is obtained in Theorem \ref{prop2} and Lemma~\ref{le1}. 

$(iv)\Rightarrow (vii)$. Let us show first that
\begin{equation}
\label{eq9a}
\frac12\le sF^{-1}(1/s)G^{-1}(1/s)\le 1,\;\;0<s\le 1.
\end{equation}
Indeed, on the one hand, since $\|\chi_{(0,s)}\|_{L_F}=1/F^{-1}(1/s)$ and $\|\chi_{(0,s)}\|_{L_G}=1/G^{-1}(1/s)$, then for $f=F^{-1}(1/s)\chi_{(0,s)}$, $g=G^{-1}(1/s)\chi_{(0,s)}$ we have $\|f\|_{L_F}=\|g\|_{L_G}=1$ and hence
$$
sF^{-1}(1/s)G^{-1}(1/s)=\int_0^1 f(t)g(t)\,dt\le \|f\|_{L_F}\|g\|_{L_G}=1,$$
which implies the right-hand side inequality in \eqref{eq9a}.
On the other hand, the facts that $\|\chi_{(0,s)}\|_{L_F}\|\chi_{(0,s)}\|_{L_F^*}=s$ \cite[Formula~(4.39)]{KPS} and $\|g\|_{L_G}\le \|g\|_{L_F^*}\le 2\|g\|_{L_G}$, $g\in L_G$ \cite[Formula~(9.24)]{KR} imply the left-hand side inequality. 

Let now $\{g_n\}_{n=1}^\infty$ be a normalized (in $L_G$) sequence of characteristic functions of disjoint sets $A_n\subset [0,1]$, $n=1,2,\dots$, that is, $g_n=G^{-1}(1/s_n)\chi_{A_n}$, where $s_n=m(A_n)$. If $f_n:=F^{-1}(1/s_n)\chi_{A_n}$, $n=1,2,\dots$, then by the hypothesis, there exists a subsequence $\{f_{n_k}\}\subset \{f_n\}$ such that 
\begin{equation}
\label{eq11}
\|(a_k)\|_{\ell_1}\le C\Big\|\sum_{k=1}^\infty a_kf_{n_k}\Big\|_{L_F}
\end{equation}
with some constant $C$ independent of given sets $A_n$, $n=1,2,\dots$

It is well known (see e.g. \cite[II.3.2]{KPS}) that the projection
$$
Pf:=\sum_{k=1}^\infty \frac{1}{s_{n_k}}\int_{A_{n_k}}f(t)\,dt\cdot \chi_{A_{n_k}}$$
is bounded in $L_F$ and $\|P\|_{L_F}=1$.   
Hence, from \eqref{eq11} and \eqref{eq9a} it follows
\begin{eqnarray*}
\Big\|\sum_{k=1}^\infty b_kg_{n_k}\Big\|_{L_G} &=& \sup\Big\{\int_0^1\Big(\sum_{k=1}^\infty b_kg_{n_k}(t)\Big)f(t)\,dt:\,\|f\|_{L_F}\le 1\Big\}\\
&=& \sup\Big\{\sum_{k=1}^\infty b_k G^{-1}(1/s_{n_k})\int_{A_{n_k}}f(t)\,dt:\,\|f\|_{L_F}\le 1\Big\}\\
&\le&\sup\Big\{\sum_{k=1}^\infty b_k G^{-1}(1/s_{n_k})\int_{A_{n_k}}f(t)\,dt:\,\|Pf\|_{L_F}\le 1\Big\}\\
&\le&\sup\Big\{\sum_{k=1}^\infty a_kb_k s_{n_k}F^{-1}(1/s_{n_k})G^{-1}(1/s_{n_k}):\,\|(a_k)\|_{\ell_1}\le C\Big\}\\
&\le& C\|(b_k)\|_{c_0}.
\end{eqnarray*}
Thus, each normalized sequence $\{g_n\}$ of characteristic functions of disjoint sets  contains a subsequence $\{g_{n_k}\}$ $C$-equivalent in $L_G$ to the unit vector basis of $c_0$ with a uniform constant $C$ for all $\{g_n\}$. As a result, the implication $(iv)\Rightarrow (vii)$ is proved.

$(vii)\Rightarrow (i)$. On the contrary, suppose that (i) does not hold. We need to prove that for each constant $C>0$ there are disjoint subsets $A_n$, $n=1,2,\dots$, of $[0,1]$ such that the sequence $\{g_n\}$, where $g_n:=G^{-1}(1/m(A_n))\chi_{A_n}$, does not contain any subsequence $C$-equivalent to the unit vector basis of $c_0$.

Let $C'>C$ be fixed. Since \eqref{eq8} fails for $C'$ we can find $M=M(C')$ such that for some $t_k\to\infty$ we have
$$
G(C't_k)\le MG(t_k),\;\;k=1,2,\dots$$ 
Then, for the inverse function $G^{-1}$, we obtain
\begin{equation}
\label{eq12}
C'G^{-1}(\tau_k/M)\le G^{-1}(\tau_k),\;\;k=1,2,\dots,
\end{equation}
where $\tau_k:=MG(t_k)\to\infty$ as $k\to\infty$.
Denoting $s_k:=1/\tau_k$ and passing to a subsequence, we can assume that $\sum_{k=1}^\infty s_k<1$. Let $A_k\subset [0,1]$, $k=1,2,\dots$, be disjoint sets, $m(A_k)=s_k$ and $g_k:=G^{-1}(1/s_k)\chi_{A_k}$. Assume that there is a subsequence $\{g_{k_i}\}\subset \{g_{k}\}$ such that
\begin{equation}
\label{eq13}
\Big\|\sum_{i=1}^\infty b_ig_{k_i}\Big\|_{L_G}\le C\|(b_i)\|_{c_0}.
\end{equation}
Then, in particular, for each $m\in\mathbb{N}$ we have
$$
\int_0^1 G\Big(\frac{|\sum_{i=1}^m g_{k_i}(t)|}{C'}\Big)\,dt=
\sum_{i=1}^m G(G^{-1}(1/s_{k_i})/C')s_{k_i}\le 1.$$
Combining this estimate with inequality \eqref{eq12} for $\tau_{k_i}=1/s_{k_i}$, $i=1,2,\dots,m$ and taking into account that $G$ increases, we get
$$
\sum_{i=1}^mG(G^{-1}(1/(Ms_{k_i})))s_{k_i}=\frac{m}{M}\le 1\;\;\mbox{for each}\;m=1,2,\dots,$$
which is a contradiction. Thus, \eqref{eq13} does not hold, and this  completes the proof of the implication and the theorem.
\end{proof}


The next result shows that $(1-DH)$-property is not preserved under duality in the class of Orlicz spaces.

\begin{theorem}
\label{th2}
There exists a $(1-DH)$ Orlicz space $L_\Phi$ such that the dual space $L_\Psi$ ($\Psi$ is the Young conjugate Orlicz function for $\Phi$) is not $DH$. 
\end{theorem}
\begin{proof}
Let $\Phi$ be the Orlicz function from Proposition~\ref{prop3}. Then, in view of Lemma~\ref{le2}, every disjoint normalized sequence $\{f_n\}_{n=1}^\infty$ from $L_\Phi$ contains a subsequence $\{f_{n_k}\}_{k=1}^\infty$ that is $6$-equivalent to the unit vector basis of an Orlicz space $\ell_H$ for some $H\in C_\Phi^\infty$. Since $C_\Phi^\infty\cong\{t\}$ by Proposition~\ref{prop3}, we have that $H(u)\asymp u$, $0\le u\le 1$. Therefore, $\{f_{n_k}\}$ is equivalent to the unit vector basis of $\ell_1$. This implies that $L_\Phi$ is $(1-DH)$.

On the other hand, $E_\Phi^\infty\not\equiv\{t\}$ and hence, by Theorem~\ref{th1}, $L_\Psi$ fails to be uniformly $(\infty-DH)$. Thus, according to Corollary~\ref{cor2new}, $L_\Psi$ is not $(\infty-DH)$. Moreover, arguing in the same way as in the proof of implication $(iv)\Rightarrow (vii)$ of Theorem~\ref{th1}, one can readily check that $L_\Psi$ is restrictive $(\infty-DH)$. Combining these facts, we conclude that $L_\Psi$ fails to be $DH$.
\end{proof}

We restate now Theorem~\ref{th2} in the form, which is in a sharp contrast with Corollary~\ref{cor2new}, showing an essential difference between the $(1-DH)$- and $(\infty-DH)$-properties in Orlicz spaces.  

\begin{corollary}
\label{th3}
There exists a $(1-DH)$ Orlicz space $L_\Phi$, which fails to be uniformly $(1-DH)$.

Moreover, there is no constant $c$ such that for every sequence of disjoint sets $E_n\subset [0,1]$, $n=1,2,\dots$, there is a subsequence $E_{n_k}$, $k=1,2,\dots$, satisfying the inequality 
\begin{equation}
\label{eq1new}
\Big\|\sum_{k=1}^m \bar{\chi}_{E_{n_k}}\Big\|_{L_\Phi}\ge cm,\;\;m=1,2,\dots
\end{equation}
\end{corollary}
\begin{proof}
Taking for $L_\Phi$ the same space as in Theorem~\ref{th2} and applying Theorem~\ref{prop2}, we obtain immediately both assertions of the corollary.
\end{proof}

\begin{rem}
Condition \eqref{eq8} is known already a rather long time (see \cite{A19c}).
Proposition~5.8 from \cite{AKS} (see also the survey \cite[Theorem~3.4]{A19c}) reads that it holds if and only if the Orlicz space $L_F$ is $(1-DH)$, where $F$ is the Young conjugate function for $G$. In fact, by Theorem~\ref{th1}, condition \eqref{eq8} is equivalent to the uniform $(1-DH)$-property of the space $L_F$. Therefore, Corollary~\ref{th3} 
shows that \cite[Proposition~5.8]{AKS} is not true in full generality. From this discussion it follows that condition \eqref{eq8} is fulfilled if and only if the Orlicz space $L_F$ satisfies some uniform version of the so-called Dunford-Pettis criterion of weak compactness (see \cite{A19c}). 
We plan to study this issue in a next paper.


\end{rem}

Proceed now with the reflexive case. Recall that a symmetric space $X$ is said to be a disjointly complemented (in brief, $DC$) if every disjoint sequence in $X$ has a subsequence whose span is complemented in $X$.

\begin{theorem}
\label{th1new}
Let $1<p<\infty$, $1/p+1/q=1$, and let $F$ and $G$ be mutually Young conjugate Orlicz functions. Suppose that $L_F$ is a $(p-DH)$ space. Then the following conditions are equivalent:

(i) $L_G$ is $(q-DH)$;

(ii) $L_G$ is a $DC$ space;

(iii) $L_F$ is uniformly $(p-DH)$;

(iv) $E_F^\infty\equiv\{t^p\}$.
\end{theorem}
\begin{proof}
First, from \cite[Theorem~4.5]{FHT-survey} it follows that conditions $(i)$ and $(ii)$ are equivalent.

$(i)\Rightarrow (iii)$. 
By the hypothesis and Proposition~\ref{prop5}, there are constants $C_1$ and $C_2$ such that every disjoint normalized sequence in $L_F$ (resp. $L_G$) contains a subsequence admitting the $C_1$-upper $\ell_p$-estimate (resp. $C_2$-upper $\ell_q$-estimate). 

Let $\{f_n\}_{n=1}^\infty$ be any disjoint normalized sequence in $L_F$. As above, we find a disjoint sequence $\{g_n\}_{n=1}^\infty$ from $L_G$, $\|g_n\|_{L_G}=1$, such that $\int_0^1 f_n(t)g_m(t)\,dt=\delta_{n,m}$, $n,m=1,2,\dots$. If $\{f_{n_k}\}\subset \{f_n\}$ is a subsequence that admits the $C_2$-upper $\ell_q$-estimate, then we have
\begin{eqnarray*} 
\Big\|\sum_{k=1}^\infty a_kf_{n_k}\Big\|_{L_F} &\ge& \sup\Big\{\int_0^1\Big(\sum_{k=1}^\infty a_kf_{n_k}(t)\Big)\Big(\sum_{k=1}^\infty b_kg_{n_k}(t)\Big)\,dt:\,\Big\|\sum_{k=1}^\infty b_kg_{n_k}\Big\|_{L_G}\le 1\Big\}\\
&\ge& \sup\Big\{\sum_{k=1}^\infty a_kb_k:\,\|(b_k)\|_{\ell_q}\le C_2^{-1}\Big\}=C_2^{-1}\|(a_k)\|_{\ell_p}.
\end{eqnarray*}
Thus, every disjoint normalized sequence in $L_F$ contains a subsequence satisfying the $C_2^{-1}$-lower $\ell_p$-estimate. This completes the proof of the implication.

Since the implication $(iii)\Rightarrow (iv)$ is an immediate consequence of Theorem~\ref{prop2}, we need only to prove that $(iv)$ implies $(i)$. 

Let us assume that $E_F^\infty\equiv\{t^p\}$. Then, arguing in the same way as in the proof of the implication $(iv)\Rightarrow (vii)$ of Theorem~\ref{th1}, we get that $E_G^\infty\equiv\{t^q\}$. Combining this fact with Theorem~\ref{prop2}, we obtain that $L_G$ is uniformly $(q-DH)$.  
\end{proof}

\begin{rem}
Following to \cite{FHSTT}, we say that a Banach lattice $X$ with a basis $\{x_n\}$ such that $\sum_{n=1}^\infty a_nx_n\ge 0$ if and only if $a_n\ge 0$ for all $n$ is a {\it Banach lattice ordered by basis}. Clearly, if $X$ is ordered by basis, then the vectors $x_n$, $n=1,2,\dots$, are disjoint. In \cite[Proposition~6]{FHSTT}, it is proved that if such a Banach lattice $X$ is $(p-DH)$ and $X^*$ is $(q-DH)$ ($1<p<\infty$, $1/p+1/q=1$), then $X$ is uniformly $(p-DH)$. Theorem~\ref{th1new} indicates that in the case of Orlicz spaces analogous result holds together with its converse for "usual"\:a.e. order of functions.
\end{rem}

Now we are able to show now that $(p-DH)$-property fails to be preserved under duality in the class of Orlicz spaces also in the case when $1<p<\infty$ (cf. \cite[p.~5863]{FHSTT} and \cite[p.~6]{FHT-survey}, where the opposite is asserted; see Section~\ref{Intro}). 

\begin{theorem}
\label{th2new}
For every $1<p<\infty$ there exists a $(p-DH)$ Orlicz space $L_{\Phi_p}$ such that the dual space $L_{\Psi_q}$, where $\Psi_q$ is the Young conjugate Orlicz function for $\Phi_p$, is restricted $(q-DH)$ but is not $DH$ and not $DC$. 
\end{theorem}
\begin{proof}
Let $\Phi_p$ be the Orlicz function from Corollary~\ref{cor3new}. First, by Lemma~\ref{le2}, every disjoint normalized sequence $\{f_n\}_{n=1}^\infty$ from $L_{\Phi_p}$ contains a subsequence $\{f_{n_k}\}_{k=1}^\infty$ $6$-equivalent to the unit vector basis of an Orlicz space $\ell_H$ for some $H\in C_{\Phi_p}^\infty$. Since $C_{\Phi_p}^\infty\cong\{t^p\}$ by Corollary~\ref{cor3new}, then $H(u)\asymp u^p$, $0\le u\le 1$. Therefore, $\{f_{n_k}\}$ is equivalent to the unit vector $\ell_p$-basis. This implies that $L_{\Phi_p}$ is $(p-DH)$. In particular, $L_{\Phi_p}$ is restricted $(p-DH)$, and in the same way as in the proof of the implication $(iv)\Rightarrow (vii)$ of Theorem~\ref{th1},
it can be shown that the dual space $L_{\Psi_q}$ is restricted $(q-DH)$ as well (see also \cite[Proposition~3.7]{HST}).

On the other hand, since $E_{\Phi_p}^\infty\not\equiv\{t^p\}$, from Theorem~\ref{th1new} it follows that $L_{\Psi_q}$ is not $(q-DH)$. Since $L_{\Psi_q}$ is restricted $(q-DH)$, this yields that $L_{\Psi_q}$ fails to be $DH$.
\end{proof}

Theorem 4.1 in the paper  \cite{bib:04} reads that an Orlicz space $L_F$ is $DH$ if and only if  $E_{F}^\infty\cong\{\varphi\}$ for a certain function $\varphi$. Theorem 5.1 in \cite{HST}, the proof of which is based on the above result, claims that every restricted $(2-DH)$ Orlicz space is $(2-DH)$. However, combining Theorem~\ref{th2new} with Proposition~\ref{prop1}, we obtain the following completely different result.

\begin{corollary}
\label{restricted and not}
Let $1<q<\infty$. There exists an Orlicz function $\Psi_q$ such that $E_{\Psi_q}^\infty\cong\{t^q\}$ and the Orlicz space $L_{\Psi_q}$ is restricted $(q-DH)$, but $C_{\Psi_q}^\infty\not\cong\{t^q\}$ and $L_{\Psi_q}$ fails to be $(DH)$.
\end{corollary}

\end{document}